\documentclass[12pt,reqno,oneside]{amsart}
\usepackage{amsmath,amsthm,amsfonts,amssymb}
\usepackage{indentfirst}
\usepackage{url}
\usepackage{hyperref}
\usepackage{graphicx}

\theoremstyle{plain}
\newtheorem{teo}{Theorem}[section]
\newtheorem{cor}[teo]{Corollary}
\newtheorem{lem}[teo]{Lemma}
\newtheorem{prop}[teo]{Proposition}

\theoremstyle{definition}

\newtheorem{exa}[teo]{Example}
\newtheorem{obs}[teo]{Remark}

\numberwithin{equation}{section}

\newcommand{\bbZ}{\mathbb{Z}}
\newcommand{\bbN}{\mathbb{N}}
\newcommand{\bbP}{\mathbb{P}}

\newcommand{\cA}{\mathbb{A}}
\newcommand{\cM}{\mathbb{M}}
\newcommand{\eps}{\varepsilon}
\newcommand{\nseta}{\nrightarrow}

\newcommand{\surv}{\textit{survives w.p.p.}}
\newcommand{\dies}{\textit{dies out a.s.}}

\newcommand{\qn}{(q_n)_{n=1}^{\infty}}
\newcommand{\qnk}{(q_{n_k})_{k=1}^{\infty}}
\newcommand{\process}{\Gamma[N,L,(q_n)_{n=1}^{\infty}]}
\newcommand{\an}{a_n(N,L,(q_n)_{n=1}^{\infty})}

\newcommand{\arn}{a_{r_n}(N,L,(q_n)_{n=1}^{\infty})}
\newcommand{\arnImpar}{a_{r_{2n-1}}(N,L,(q_n)_{n=1}^{\infty})}

\begin{document}

\baselineskip=18pt

\title{Random walks systems with finite lifetime on $ \bbZ $}

\author[1]{Elcio~Lebensztayn}
\author[2]{F\'abio~Prates~Machado}
\author[3]{Mauricio~Zuluaga~Martinez}


\address[Elcio~Lebensztayn]
{Institute of Mathematics, Statistics and Computer Sciences - UNICAMP - Brazil}
\address[F\'abio~Prates~Machado]
{Institute of Mathematics and Statistics - USP - Brazil}
\address[Mauricio~Zuluaga~Martinez]
{UFERSA - Brazil}

\email{lebensztayn@ime.unicamp.br, fmachado@ime.usp.br}
\email{zuluagamm@gmail.com}

\thanks{Research supported by CNPq (303872/2012-8 and 310829/2014-3) and FAPESP(09/52379-8 and 12/22673-4)}


\date{\today}

\begin{abstract}
We consider a non-homogeneous random walks system on $\bbZ$ in which each active
particle performs a nearest neighbor random walk and activates
all inactive particles it encounters up to a total amount of $L$ jumps. We present 
necessary and sufficient conditions for the process to survive, which means that an 
infinite number of random walks become activated.
\end{abstract}

\maketitle

\section{Introduction}
\label{S: Introduction}

We study conditions for extinction or survival for a non-homogeneous 
random walks system on $\bbZ$ such that at time zero, starts from $N$ 
particles at each vertex of $\bbN = \{ 1, 2, \dots \}$. All
particles are inactive, except for those placed at the vertex~1.
The active particles move as discrete-time independent non-homogeneous 
nearest neighbor random walks on~$\bbZ$, activating the inactive particles
they encounter along their way up to their $L-th$ step. 
We suppose that the jump probabilities of the active particles depend 
on their initial position: particles initially placed at position~$n$, 
if activated, jump $L$ random steps being each of them to the right with 
probability $1 - q_n$ or to the left with probability $q_n$. We refer to
the process as $\process$.

The model can be thought of as for the evolution of a disease, by contact, 
in a population of connected individuals living in small colonies. Each infected
individual helps spreading the disease up to its first $L$ (random) steps.
The main purpose is to study whether the
process survives (globally), that is, whether there is a positive probability
that an infinite number of individuals become infected. The answer depends on
the trio $(N,L,\qn)$.
In Lebensztayn et al.~\cite{RWSZ,NHRWSZ}, the survival issue is investigated for similar models.
Kurtz et al.~\cite{CG} study this model on the complete graph and state limit theorems for the proportion of visited vertices. The question of local survival (infinite number of visits of active particles to the origin) in a model on~$\bbZ$ is studied by Gantert and Schmidt~\cite{GS}. For problems related to global survival of branching random walks on graphs, we refer to Bertacchi and Zucca~\cite{BZ} and to Bertachi {\it et al}~\cite{BMZ}. They also face questions related to local survival.

\section{Main Results}


For each sequence $\qn$, let
\begin{equation*}
m(\qn)= \min \biggl\{M \in \bbN: \sum_{n=1}^{\infty} (q_n)^M < \infty \biggr\}.
\end{equation*}
Observe that $m(\qn)=\infty$ if and only if $\sum_{n=1}^{\infty} (q_n)^M=\infty$ for every $M>0$.
We also consider the following integers
\begin{equation}
\label{F: a}
b(N,L):=\begin{cases}
N (\frac{L+1}{2})^2 &\text{if } L \text{ is odd}, \\[0.2cm] 
N \frac{L(L+2)}{4} &\text{if } L \text{ is even}.
\end{cases}
\end{equation}

We highlight two sets of sequences of probabilities

\begin{align*}
D &:= \left\{(q_n)_{n=1}^\infty:m((q_n)_{n=1}^\infty)=\infty \right\}, \\[0.1cm]
D_1 &:= \left\{(q_n)_{n=1}^\infty:q_1 \geq q_2 \geq \dots \right\}. \\[0.1cm]
\end{align*}

Our main goal is to be able to tell for a large set of parameters 
$(N,L,(q_n)_{i=1}^{\infty})$
if \textit{the process survives with positive probability} (we write $\process$ $\surv$)
or \textit{the process dies out almost surely} (we write $\process$ $\dies$). With this in mind
we present the first results.

\eject

\begin{teo}
\label{T: OnC}
Let $\tilde{q}:=(q_n)_{n=1}^\infty \in D^c$.
\begin{itemize}
\item[\textit{(a)}] If $m(\tilde{q}) > b(N,L)$ and $\tilde{q} \in D_1$, then $\Gamma(N,L,\tilde{q})$ dies out a.s.
\item[\textit{(b)}] If $m(\tilde{q}) \leq b(N,L)$, then $\Gamma(N,L,\tilde{q})$ survives w.p.p.
\end{itemize}
\end{teo}

\begin{teo}
\label{T: OnD}
Let $\tilde{q}:=(q_n)_{n=1}^\infty \in D$.
\begin{itemize}
\item[\textit{(a)}] If $\tilde{q} \in D_1 $ then $\Gamma(N,L,\tilde{q}) \ \dies$
\item[\textit{(b)}] If $\tilde{q}$ is such that for every subsequence $(q_{n_k})_{k=1}^\infty \in D^c$, $\{n_{k+1}-n_k\}_{k \geq 1}$ is unbounded, then $\Gamma(N,L,\tilde{q})\ \dies$
\item[\textit{(c)}] If there exists a subsequence $ (q_{n_k})_{k=1}^\infty \in D^c$ such that $\{n_{k+1}-n_k\}_{k \geq 1}$ is bounded, then $\Gamma(N,L,\tilde{q})$ survives w.p.p. for large values of $N$ and $L$.
\end{itemize}
\end{teo}

\begin{obs}
\label{R: CriteriaNonIncreasing}
From Theorems~\ref{T: OnC} and~\ref{T: OnD} we can present a criteria for sequences $\qn \in D_1$
\[ \process \ \surv \iff m(\qn) \leq b(N,L).\]
\end{obs}

\begin{exa}
For $q_n = \frac{1}{\log (n+2)},\ m(\qn) = \infty$. Then $\process$ $\dies$
\end{exa}

\begin{exa}
For $q_n = \frac{1}{\sqrt n},\ m(\qn) = 3$. Then 
\[ \process \ \surv \iff  b(N,L) \geq 3. \]
\end{exa}

\begin{exa}
Consider
\begin{equation*}
q_n:=\begin{cases}
\frac{1}{k} &\text{if } n = 2^k, \\[0.2cm] 
\frac{1}{\log (n+2)} &\text{otherwise }.
\end{cases}
\end{equation*}
Then $\process$ $\dies$ by part (b) of Theorem~\ref{T: OnD}.
\end{exa}

\begin{exa}
Consider
\begin{equation*}
q_n:=\begin{cases}
\frac{1}{k} &\text{if } n = 2 k, \\[0.2cm] 
\frac{1}{\log (k+2)} &\text{if } n=2k+1.
\end{cases}
\end{equation*}
Then $\process$ $\surv$ for $N$ and $L$ large enough by part (c) of Theorem~\ref{T: OnD}. To figure out how large $N$ and $L$ must be see Theorem~\ref{T: Final}.
\end{exa}



\section{Auxiliary results}

We start off with a basic but useful result.
The proof can be found in~Bremaud~\cite[p.~422]{Bremaud}.

\begin{lem}
\label{L: IP} 
Let $\{a_n\}_{n \geq 1}$ be a sequence of real numbers in $(0,1)$. Then, 
$$ \prod_{n=1}^\infty{(1-a_n)}=0 \Longleftrightarrow \sum_{n=1}^\infty a_n=\infty. $$
\end{lem}

\begin{exa}
For $L=1$ and 
\begin{equation*}
q_n:=\begin{cases}
\frac{1}{\log k} &\text{if } n = 2^k, \\[0.2cm] 
\frac{1}{n} &\text{otherwise },
\end{cases}
\end{equation*} 
${\Gamma[N,1,(q_n)_{n=1}^{\infty}]} \ \dies$. Observe that as $L=1$ the probability for the
process to survive is $\prod_{i=1}^{\infty} (1-(q_i)^N)$. From Lemma~\ref{L: IP} follows that the later product equals 0 as $\sum_{i=1}^{\infty} (q_i)^N = \infty $ for all $N.$  For $L>1$ see Theorem~\ref{T: Final}.
\end{exa}


Some definitions are needed to proceed. At time zero, at each site of 
$\bbN = \{ 1, 2, \dots \},$ there are $N$ particles that are able, once activated, to perform
discrete-time independent non-homogeneous nearest neighbor random walks on $ \bbZ $. We
can define a product space in such a way that all these walks are prescribed from the 
begining. We define

\begin{align*}
S_n^j(i): & \text{ the } j-th  \text{ random walk, } j = 1, \dots, N, \text{ starting from vertex } i \\[0.1cm]
& \text{ that at each step goes to the left with probability } q_i \text{ or to the  }  \\[0.1cm] 
& \text{ right with probability } 1-q_i. \text{ Here } 0 \le n \le L.  \\[0.1cm]
R_i^j &:= \{S_n^j(i) : 0 \le n \le L \}, \ \text{the virtual range of the } j-th \text{ particle }\\[0.1cm]
& \text{originally placed at site } i. \text{ It is prescribed from the beggining but}\\[0.1cm]  & \text{these random walks are activated only if vertex i is visited by an}\\[0.1cm]
& \text{active particle.} \\[0.1cm]
\end{align*}

\begin{align*}
R_i &= \bigcup_{j=1}^N R_i^j, \text{ the set of vertices prescribed to be visited by some of} \\[0.1cm]
& \text{the }N \text{ random walks of the particles originally placed at site } i.
\end{align*}

Next we define the event $\{i \rightarrow j\}$ as the event such that the vertex $j$ is 
prescribed to be visited by some of the $N$ random walks of the particles originally 
placed at site $i$. Observe that $\{i \rightarrow j\} = \emptyset$ if $j-i > L.$ Besides, $\{i \nseta j\}:=\{i \rightarrow j\}^c.$

Analogously the event $\{i \leadsto j\}$ holds if and only if there exists an finite sequence of distinct vertices $ i = i_0, i_1, i_2, \dots, i_m = j $ such that, for all $n = 0, \dots, m-1$,
\[ i_{n+1} \in R_{i_n}. \]


With this last definition in mind, notice that $\process \ survives $ for some particular realization of the process, if and only if there exists an infinite sequence of distinct vertices $ 0 = i_0, i_1, i_2 \dots $ such that, for all $n$
\[ i_{n+1} \in R_{i_n}. \]

For $ i \in \bbN $, we define the event $ E_i := \{1 \leadsto i\} $ which in words is the event
where the particles at vertex $i$ are activated, as all particles originally placed at vertex $1$ are
activated from the begining.

%

For what follows it is important to realize that the event $E_i$ and the event $\{i \rightarrow j\}$ are
independent as each of them depends on the displacements of different sets of random walks.

Now we prove some auxiliary results that will lead us to the proofs of Theorems~\ref{T: OnC} and~\ref{T: OnD}. In order to do that, first consider blocks of size $L$ like

\[ \cA_n = \{ n+1, n+2,\dots,n+L \}. \]
 
Besides, for any sequence $(r_n)_{n=1}^{\infty} \in \bbN$, such that 
$r_{n+1}-r_n \geq L$, we highlight a special subset of $\cA_n,$

\[ \cA_{r_n} = \{ r_n+1,\dots,r_n+L \}. \]

Notice that the particles in $n+L+1$ can only be activated by some particle whose
original position is in $\cA_n$.
We denote the probability that the vertex $n +L+ 1$ is not visited by some particle initially placed in $\cA_n$ by

\[ \an:=\prod_{i=n+1}^{n+L} \bbP(i \nseta n+L+1). \]

Analogously we denote the probability that the vertex $(r_n+L+1)$ is not visited by some particle initially placed in $\cA_{r_n}$ by

$$ \arn:=\prod_{i=r_n+1}^{r_n+L} \bbP(i \nseta r_n+L+1). $$

Our first auxiliary result gives a simple condition for the almost sure extinction of the process.

\begin{prop} We have that
\label{P: Extinction}
\[ \sum_{n=1}^{\infty} \arn = \infty \Rightarrow \process \ \dies \] 
\end{prop}

\begin{proof}
Let $B_n = \bigcap_{j=1}^{L} \{r_n+j \nseta r_n+L+1 \}$. 
If $B_n$ occurs for some $n$, then the vertex $(r_n+L+1)$ would not be visited and the process would die out a.s.\ 
Now, as
$$ \sum_{i=1}^{\infty} \bbP(B_n) = \sum_{n=1}^{\infty} \arn = \infty, $$
we obtain, from Borel-Cantelli Lemma, that $\bbP(B_n \text{ infinitely often}) = 1$.
This implies the result.
\end{proof}

We denote the integer part function by $\lfloor \cdot \rfloor$.
For $j \in \{1,2,\dots,L\}$, let
$$ f(j) := \Big\lfloor \frac{j+1}{2}\Big\rfloor. $$
The next proposition will be useful in order to obtain a new condition for extinction (to be derived from Proposition~\ref{P: Extinction}).

\begin{prop}
\label{P: Prob geq}
Let $j \in \{1,2,\dots,L\}$.
We have that
\begin{equation}
\label{E: Prob geq}
\bbP(r_n+j \nseta r_n+L+1) \geq (q_{r_n+j})^{N f(j)}.
\end{equation}
\end{prop}

\begin{proof}
Observe that, if every particle originally placed at vertex $(r_n+j)$ would make its first $f(j)$ 
jumps to the left, then $\{r_n+j \nseta r_n+L+1\}$ would occur.
This happens since, in the best case, this particle makes $(L-f(j))$ jumps to the right, thus the rightmost vertex it would visit during its life would be 
\begin{align*}
s &:= r_n+j+(L-f(j))-f(j)< r_n+L+1 .
\end{align*}
%
Hence,
\[ \bbP (r_n+j \nseta r_n+L+1) \geq (q_{r_n+j})^{N f(j)} . \]
\end{proof}

Now it is easy to see that
\begin{equation*}
\sum_{j=1}^{L} f(j) = \begin{cases}
(\frac{L+1}{2})^2 &\text{if } L \text{ is odd}, \\[0.2cm] 
 \frac{L(L+2)}{4} &\text{if } L \text{ is even}.
\end{cases}
\end{equation*}
Recalling~\eqref{F: a}, we have that $b(N,L)=N\sum_{j=1}^{L} f(j)$.
Consequently,

\begin{prop}
\label{P: dies}
Let $\bar{q}_n= \min \{q_{nL+1},\dots,q_{nL+L} \}$. 
\[ \sum_{n=0}^{\infty} (\bar{q}_n)^{b(N,L)}=\infty \Rightarrow \process \ \dies \]
\end{prop}

\begin{proof}
It follows from Proposition~\ref{P: Prob geq} and then from Proposition~\ref{P: Extinction}, since
\begin{align*}
\sum_{n=0}^{\infty} a_{nL}(N,L,(q_n)_{n=1}^{\infty}) &\geq \sum_{n=0}^{\infty} 
(q_{nL+1})^{N f(1)}(q_{nL+2})^{N f(2)}\dots(q_{nL+L})^{N f(L)} \\
 &\geq \sum_{n=0}^{\infty} (\bar{q}_n)^{b(N,L)}=\infty.
\end{align*}
\end{proof}

The following result (whose proof we omit) together with Proposition~\ref{P: dies} 
are the keys for understand what sequences $\qn \in D_1$ that will
lead $\process$ to extintion almost surely.

\begin{lem}
\label{L: S infty}
Let $\qn \in D_1$ and define $a:=b(N,L)$.
Define $S:=\sum_{n=1}^{\infty} (q_n)^a $ and 
$S_j:=\sum_{n=0}^{\infty} (q_{nL+j})^a$, $j \in \{1,2,\dots,L\}$.
If $S=\infty$, then $S_j=\infty$ for all $j$.
\end{lem}

To finish the section, we state two auxiliary results concerning the survival of the process.

\begin{prop}
\label{P: Survival}
We have that
\[ \sum_{n=0}^{\infty} \an < \infty \Rightarrow \process \ \surv \]
\end{prop}

\begin{proof} 
Notice that
\begin{align*}
E_{n+L+1} &= (E_{n+L} \cap \{n+L \rightarrow n+L+1 \}) \cup \\[0.1cm] &\bigcup_{i=n+1}^{n+L-1} 
\Big(E_i \cap \{i \rightarrow n+L+1 \} \cap \bigcap_{j=i+1}^{n+L} \{j \nseta n+L+1\}\Big).
\end{align*}
Observing that $\bbP(E_n)$ is non-increasing in $n$, we obtain the following telescopic sum
\begin{align*}
\bbP(E_{n+L+1}) &\geq \bbP(E_{n+L}) \Big[\bbP(n+L \rightarrow n+L+1)\\[0.1cm]
&+\sum_{i=n+1}^{n+L-1} \bbP(i \rightarrow n+L+1) \prod_{j=i+1}^{n+L} \bbP(j \nseta n+L+1)\Big]\\[0.1cm]
&= \bbP(E_{n+L})\Big[1-\prod_{i=n+1}^{n+L} \bbP(i \nseta n+L+1)\Big].
\end{align*}
Iterating this formula $n$ times
\begin{align*}
\bbP(E_{n+L+1}) \geq \bbP(E_{L}) \prod_{k=1}^n \Big[1-\prod_{i=k+1}^{k+L} \bbP(i \nseta k+L+1)\Big].
\end{align*}

Now passing to the limit, using the hypotheses and Lemma~\ref{L: IP} we have that
\[ \prod_{n=1}^\infty\Big[1-\prod_{i=n+1}^{n+L} \bbP(i \nseta n+L+1)\Big]>0. \]
From this we conclude that the process has positive probability of survival.
\end{proof}

\begin{prop}
\label{P: Prob leq}
Let $j \in \{1,2,\dots,L\}$.
\begin{equation}
\label{E: Prob leq}
\bbP(n+j \nseta n+L+1) \leq 2^{N L}(q_{n+j})^{N f(j)}.
\end{equation}
\end{prop}

\begin{proof}
First notice that, if $e$ is an integer and $e<f(j)$, then 
$$1 \leq j-2e.$$
Indeed, $e<f(j) \leq (j+1)/2$ implies that $2e<2f(j) \leq j+1$, thus
$$2 \leq 2(f(j)-e) \leq j-2e+1.$$

The proposition is an immediate consequence of the following claim: if a particle at vertex $(n+j)$, in its virtual trajectory, does not visit
the vertex $(n+L+1)$, then it makes at least $f(j)$ jumps to the left. 
To prove this claim, we observe that, if a particle at vertex $(n+j)$ makes $e$ jumps to the left, then it makes $(L-e)$ jumps to the right, so at the end of its $L$ jumps it will be at vertex 
$$ s:=(n+j)+(L-e)-e=n+L+j-2e. $$
Now if $e<f(j)$, then $s \geq n+L+1$ and consequently the particle would visit the vertex $(n+L+1)$.
Then 
\[ \bbP(n+j \nseta n+L+1) \leq 2^{N L} (q_{n+j})^{N f(j)}.\]
\end{proof}

\section{Proofs of main results}

%
%
%
%

\begin{proof}[Proof of Theorem~\ref{T: OnC} (a)]
Since $\qn \in D_1$ it follows that 
\[q_{(n+1)L}= \min \{q_{nL+1},\dots,q_{(n+1)L} \}.\]

As $b(N,L)<m(\qn)$, we have that $S=\sum_{n=1}^{\infty} (q_n)^{b(N,L)}=\infty$.
Hence, from Lemma~\ref{L: S infty}, $S_L=\sum_{n=1}^{\infty} (q_{nL})^{b(N,L)}=\infty$.
The result follows from Proposition~\ref{P: dies}.
\end{proof}

\begin{proof}[Proof of Theorem~\ref{T: OnC} (b)]
It follows from Propositions~\ref{P: Survival} and~\ref{P: Prob leq}.
Indeed, if $\qn \in D^c$ and $m(\qn) \leq b(N,L) $, then
\begin{align*}
\sum_{n=0}^{\infty} \an &\leq 2^{N L^2} \sum_{n=0}^{\infty} 
(q_{n+1})^{N f(1)}\times \dots \times (q_{nL+L})^{N f(L)} \\
 &\leq 2^{N L^2} \sum_{n=0}^{\infty} \sum_{j=1}^L (q_{n+j})^{b(N,L)} < \infty,
\end{align*}
therefore the process survives with positive probability.
\end{proof}

\begin{proof}[Proof of Theorem~\ref{T: OnD} (a)]
Observe that the reasoning presented in the proof of part \textit{(a)} of Theorem~\ref{T: OnC} 
can be used here, even if $m=\infty$.
\end{proof}

\begin{proof}[Proof of Theorem~\ref{T: OnD} (b)]
Let $\bar{q}_n= \min \{q_{nL+1},\dots,q_{nL+L} \}$. 
By the hypotheses, $\sum_{n=0}^{\infty} (\bar{q}_n)^{b(N,L)} <\infty$ can not hold.
Hence, $\sum_{n=0}^{\infty} (\bar{q}_n)^{b(N,L)}=\infty$ and from Proposition~\ref{P: dies} the process dies out a.s for all $L, N$.
\end{proof}

\begin{proof}[Proof of Theorem~\ref{T: OnD} (c)]
Let $L \geq A :=\max_k \{ n_{k+1}-n_k \}$ and $N \geq m(q_{n_k})$. 
By hypothesis $A$ and $N$ are finite. If for each $k$ there exists at least one particle at $n_k$ which
gives $L$ successive jumps to its right then $E_n$ holds for all $n$. Now 
$\sum \an < \infty$ as 
\begin{align*}
\an &< 2^{N L^2} (q_{n+1})^{N f(1)}\times \dots \times (q_{nL+L})^{N f(L)} \\ 
&< 2^{N L^2} (q_{n_k})^N \times (q_{n_{k + 1}})^N \times \cdots \\
&< L 2^{N L^2} (q_{n_k})^N
\end{align*}
(if $a,b \in [0,1]$ then $a \times b < a + b $) and the process survives by Proposition~\ref{P: Survival}.
\end{proof}

\section{Final Remarks}

The following results help to figure out what happens when $\qn \notin D_1$ besides
making clear the influence of the parameters $N$ and $L$.

\begin{obs}
\label{R: CriteriaSerie-an}
Observe that we are now able to present the following criteria 
\begin{equation}
\label{E: CriteriaSerie-an}
\sum_{n=0}^{\infty} \an = \infty \iff \process \ \dies
\end{equation}

The result follows from Proposition~\ref{P: Survival} and Proposition~\ref{P: Extinction} 
with $r_n=nL+1+k$ for $\{k=0,1,...,L-1\}$.
\end{obs}

Given a subsequence $(q_{n_k})_{k=1}^\infty$ of a sequence $(q_n)_{n=1}^\infty$, let
\[ \cM = \{n_{k+1}-n_k: k \geq 1 \}. \]
If $\cM$ is bounded, then we define
\[ l((q_{n_k})_{k=1}^\infty) :=  \max \{ m \in \cM: m \text{ appears infinitely often in } \cM \}, \]
otherwise let
\[ l((q_{n_k})_{k=1}^\infty) := \infty. \]

For any sequence $(q_n)_{n=1}^\infty$, we define
\begin{align*}
L_0 &:= L_0((q_n)_{n=1}^\infty)= \min \biggl\{l((q_{n_k})_{k=1}^\infty): (q_{n_k})_{k=1}^\infty \in D^c \biggr\}, \\[0.1cm]
L_1 &:= L_1((q_n)_{n=1}^\infty)= \min \biggl\{l((q_{n_k})_{k=1}^\infty) \times m((q_{n_k})_{k=1}^\infty): (q_{n_k})_{k=1}^\infty \in D^c \biggr\}.
\end{align*}
Observe that $L_0>1$ if and only if $(q_n)_{n=1}^\infty \in D$.
Moreover, $L_0=\infty$ if and only if $\{n_{k+1}-n_k\}_{k \geq 1}$ is unbounded for every subsequence $(q_{n_k})_{k=1}^\infty \in D^c$.

Observe that $l((q_n)_{n=1}^\infty)=1$ for any sequence. Besides that, for any sequence
$(q_n)_{n=1}^\infty \in D_1$, it holds that $L_1((q_n)_{n=1}^\infty)=m((q_n)_{n=1}^\infty)$. 

To see the later, consider first the case of a subsequence such that $ l((q_{n_k})_{k=1}^\infty) = \infty $ and after the case of a subsequence where $\cM$ is bounded by a finite $M$. For that sequence $m(\qnk) \geq m(\qn)$ because 
\[ \sum_{n=1}^{\infty} (q_n)^m \leq M \sum_{k=1}^{\infty} (q_{n_k})^m < \infty \]
where $m = m(\qnk).$ 

\begin{teo}
\label{T: Final}
For any sequence $(q_n)_{n=1}^\infty$, we have
\begin{itemize}
\item[\textit{(a)}] If $L < L_0$, then $\process$ \dies
\item[\textit{(b)}] If $L \in \{L_0, L_0+1,...,L_1-1\}$, then $\process$ \surv \ for all large $N$.
\item[\textit{(c)}] If $L \geq L_1$, then $\process \ \surv$
\end{itemize}
\end{teo}

\begin{proof}[Proof of Theorem~\ref{T: Final} (a)]
Suppose  that a subsequence $(q_{n_k})_{k=1}^\infty \in D^c$ of a sequence $(q_n)_{n=1}^\infty$
is such that $L_0=l((q_{n_k})_{k=1}^\infty)$. From this subsequence we pick another infinite increasing subsequence $(r_n)_{n=1}^\infty$, such that for each set of integers $\{ r_{2n-1}, r_{2n} \}$ it is true that $ r_{2n} - r_{2n-1} = L_0$.

Then there are a subsequence $(r_n)_{n=1}^\infty$ of $(n_k)_{k=1}^\infty$ such that 
\[ \{r_{2n-1}+1,r_{2n-1}+2
,...,r_{2n}-1\} \cap (n_k)_{k=1}^\infty = \emptyset. \]

For $L < L_0$ let $\bar{q}(2n-1)= \min \{q_{r_{2n-1}+1},q_{r_{2n-1}+2},\dots,q_{r_{2n-1}+L} \}$
and $n(r_{2n-1})$ the first position where $q_{n(r_{2n-1})}=\bar{q}(2n-1).$

By the definition of $L_0$, $\sum_{n=0}^{\infty} (\bar{q}(2n-1))^z =\infty$ $\forall z \in \bbN$. Otherwise with the subsequences 
$({n_k})_{k=1}^\infty$ and $(n(r_{2n-1}))_{n=0}^\infty$ we could create a new subsequence $(b_j)_{j=1}^\infty := ({n_k})_{k=1}^\infty  \cup (n(r_{2n-1}))_{n=0}^\infty $ such that 
$(q_{b_j})_{j=1}^\infty \in D^c$ and $l((q_{b_j})_{j=1}^\infty)<L_0$, which is not possible because $L_0$ is the minimum.  

Hence, $ \sum_{n=0}^{\infty} \arnImpar \geq \sum_{n=0}^{\infty} (\bar{q}(2n-1))^{b(N,L)} = \infty, $
and the result follows from Proposition~\ref{P: Extinction}.
\end{proof}

\begin{proof}[Proof of Theorem~\ref{T: Final} (b)]
Let $(q_{n_k})_{k=1}^\infty$ as defined above and
$m_0:=m((q_{n_k})_{k=1}^\infty)$. Observe that each block of size $L$ has at least one element of $(q_{n_k})_{k=1}^\infty$, therefore from

\begin{align*}
\sum_{n=0}^{\infty} a_n(m_0,L,(q_n)_{n=1}^\infty) &\leq 2^{L^2 N}\sum_{n=0}^{\infty} (q_{n+1})^{f(1)m_0}...(q_{n+L})^{f(L)m_0} \\
&\leq 2^{L^2 N}\sum_{k=1}^{\infty} (q_{n_k})^{m_0}<\infty
\end{align*}
it follows that there exists a minimal $N_0=N_0(L)$ such that
\[ \sum_{n=0}^{\infty} a_n(N_0,L,(q_n)_{n=1}^\infty)<\infty .\]

Hence the process survives with positive probability if and only if $N \geq N_0$ by~(\ref{E: CriteriaSerie-an}).
\end{proof}

\begin{proof}[Proof of Theorem~\ref{T: Final} (c)]
Observe that
\begin{align*}
\sum_{n=0}^{\infty} a_n(1,L_1,(q_n)_{n=1}^\infty) &\leq 2^{L_1^2 N}\sum_{n=0}^{\infty} (q_{n+1})^{f(1)}...(q_{n+L_1})^{f(L_1)}  \\
& \leq 2^{L^2 N} m \sum_{k=1}^{\infty} (q_{n_k})^{m}<\infty.
\end{align*}
The second inequality holds because in each block of size $L_1=l \times m$ there exists 
at least $m$ elements of $(q_{n_k})_{k=1}^\infty$. Therefore from~(\ref{E: CriteriaSerie-an}), the process survives with 
positive probability.
\end{proof}

\begin{exa}
Consider
\begin{equation*}
q_n:=\begin{cases}
\frac{1}{k} &\text{if } n = 2 k, \\[0.2cm] 
\frac{1}{\log (k+2)} &\text{if } n=2k+1.
\end{cases}
\end{equation*}
In this case $L_0=2$ and $L_1=4.$  From~(\ref{E: CriteriaSerie-an}), ${\Gamma[N,2,(q_n)_{n=1}^{\infty}]} \ \dies$ only for $N =1.$ From~(\ref{E: CriteriaSerie-an}), ${\Gamma[N,3,(q_n)_{n=1}^{\infty}]} \ \surv$ for all $N$.
\end{exa}

\section{Final Examples}
For the next two examples consider $\alpha > 0$ and the following $\qn \in D_1^c \cap D,$

\[q_{3n}=\frac{1}{n^{\alpha}}, \ q_{3n+1} = q_{3n+2} = \frac{1}{\log n}.\]

\begin{obs}
As $L_0=3$, from part (a) of Theorem~\ref{T: Final}, $\process$ $\dies$ for $L<3$.
\end{obs}

\begin{exa}
If $\alpha \geq 1$, then from~(\ref{E: CriteriaSerie-an}), the $\process \ \surv$ for $L \geq 3$.
\end{exa}

\begin{figure}[ht]
\center
\includegraphics[width=8cm]{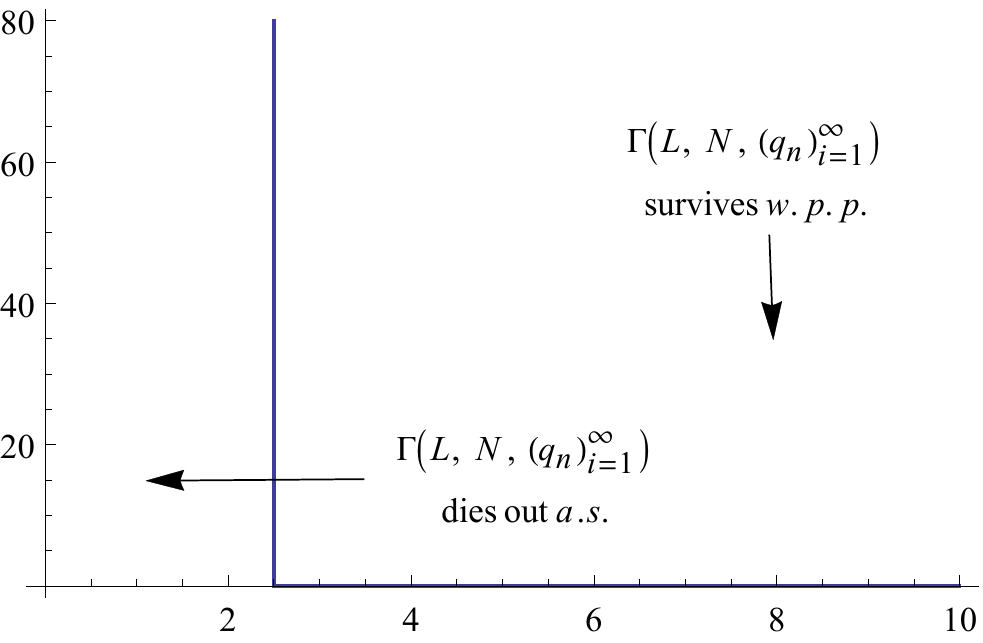}
\label{G: T6.1}
\caption{$\alpha \geq 1$}
\end{figure}
 
\begin{exa}
If $\alpha < 1$ then from~(\ref{E: CriteriaSerie-an}), for $L \geq 3$, 

\[ \process \ \dies \iff \alpha N \sum_{i=0}^{\lfloor \frac{L}{3} \rfloor -1} \lfloor \frac{3i+2}{2} \rfloor < 1.\]
\end{exa}

\begin{figure}[ht]
\center
\includegraphics[width=8cm]{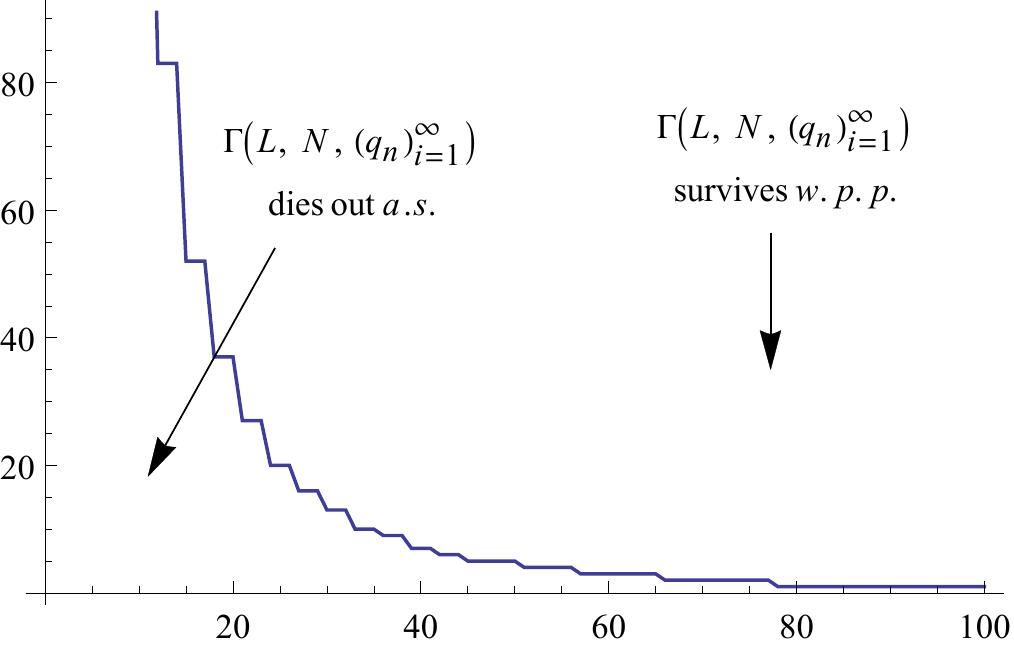}
\label{G: T6.2}
\caption{$\alpha = 10^{-3}$}
\end{figure}

\begin{exa}
From Remark~\ref{R: CriteriaNonIncreasing}, under condition that $\qn \in D_1$,
one can see the whole picture related to extinction-survival matters. In the following
picture $C=900$.

\begin{figure}[ht]
\center
\includegraphics[width=10cm]{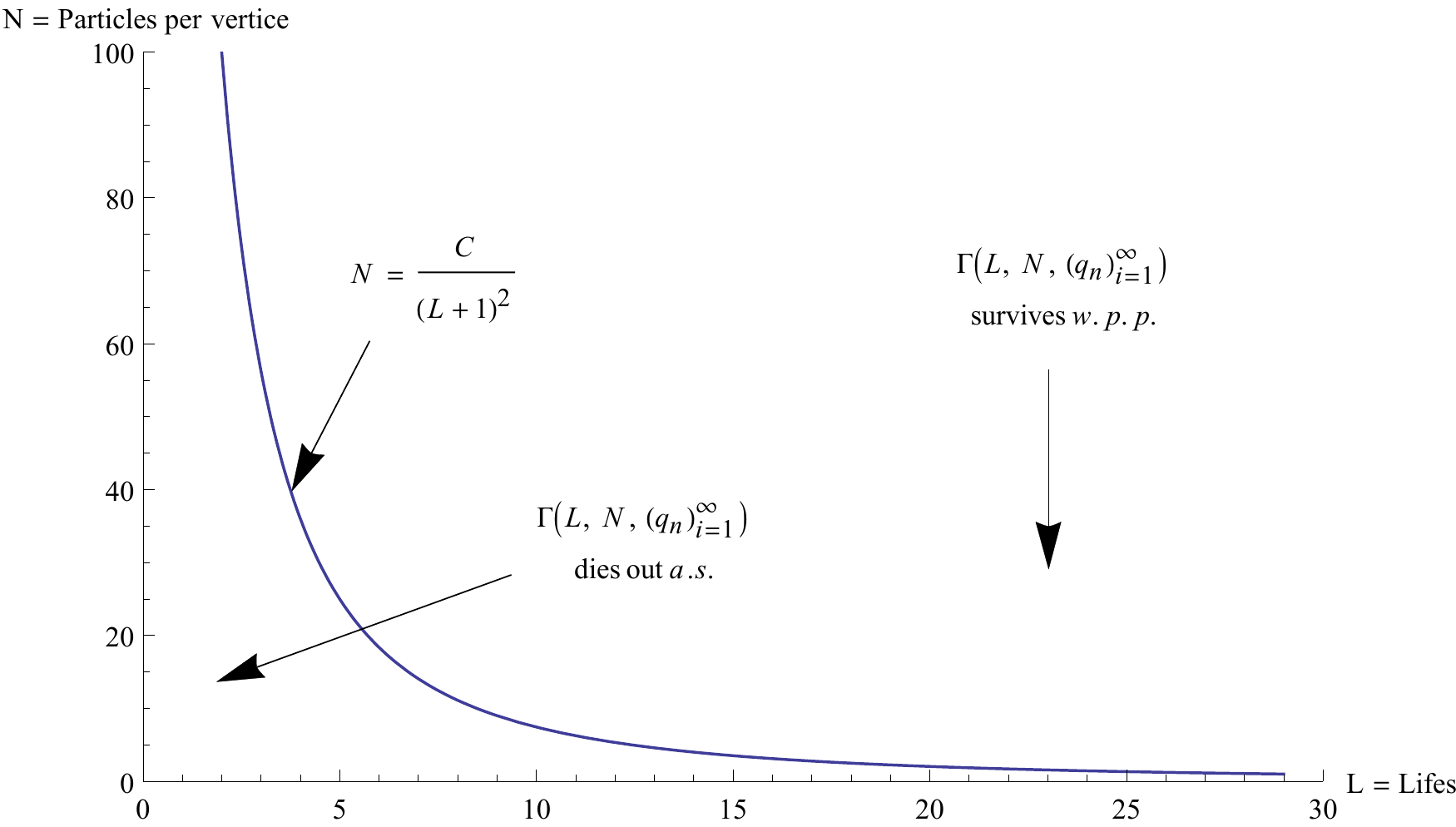}
\label{omeubrowser}
\caption{General criteria for $\qn \in D_1$}
\end{figure}

\end{exa}

\begin{exa}
Consider $0 < \alpha < 1< \beta$ and and the following $\qn \in D_1^c \cap D^c,$

\[ \ q_{3n}=q_{3n+1}= \frac{1}{n^{\alpha}}, \ q_{3n+2}= \frac{1}{n^{\beta}}.\]

Therefore from part (b) of Theorem~\ref{T: Final}, we have that $\process$ $\surv$ for
$L=3$. From~(\ref{E: Prob geq}) and~(\ref{E: CriteriaSerie-an}), we conclude that the 
following assertions hold:

For $L=2$, 
\[ \process \ \dies \iff  2 N \alpha \leq 1.\]

For $L=1$,
\[ \process \  \dies \iff    N \alpha \leq 1.\]
\end{exa}

\noindent
{\bf Acknowledgments:} 

The authors would like to thank the two anonymous reviewers for their valuable
suggestions and comments.


\begin{thebibliography}{99}

\bibitem{BZ} {D. Bertacchi and F. Zucca.}
Critical behaviors and critical values of branching random walks on multigraphs.
\textit{J. Appl. Probab.} \textbf{45}, 481--497 (2008). 

\bibitem{BMZ} {D. Bertacchi, F. Machado and F. Zucca.}
Local and global survival for nonhomogeneous random walk systems on Z.
\textit{Adv. Appl. Probab.} \textbf{46}, 256--278 (2014). 

\bibitem{Bremaud} {P. Bremaud.}
\textit{Markov chains. Gibbs fields, Monte Carlo simulation, and
queues.} 
Texts in Applied Mathematics, 31. Springer-Verlag, New York (1999).

\bibitem{GS} {N. Gantert and P. Schmidt.}
Recurrence for the frog model with drift on~$\bbZ$.
\textit{Markov Process. Related Fields} \textbf{15} (1), 51--58 (2009).

\bibitem{CG} {T. G. Kurtz, E. Lebensztayn, A. R. Leichsenring and F. P. Machado.}
Limit theorems for an epidemic model on the complete graph.
\textit{ALEA} \textbf{4}, 45--55 (2008).

\bibitem{RWSZ} {E. Lebensztayn, F. P. Machado and M. Z. Martinez.}
Random walks systems with killing on $\bbZ$.
\textit{Stochastics} \textbf{80} (5), 451--457 (2008).

\bibitem{NHRWSZ} {E. Lebensztayn, F. P. Machado and M. Z. Martinez.} 
Nonhomogeneous random walks systems on $\bbZ$. 
\textit{J. Appl. Probab.} \textbf{47} (2), 562--571 (2010).
\end{thebibliography}
\end{document}